\documentclass[notitlepage]{article}
\usepackage[utf8]{inputenc}
\usepackage{amssymb}
\usepackage{amsmath}
\usepackage{amsthm}
\usepackage{authblk}
\usepackage{xcolor}

\newtheorem{theorem}{Theorem}[section]
\newtheorem{proposition}[theorem]{Proposition}
\newtheorem{lemma}[theorem]{Lemma}
\newtheorem{definition}[theorem]{Definition}
\newtheorem{fact}[theorem]{Fact}
\newtheorem{corollary}[theorem]{Corollary}
\newtheorem{example}[theorem]{Example}

\newtheorem*{theorem*}{Theorem}

\begin{document}

\everymath{\displaystyle}

\title{Counting in Uncountably Categorical Pseudofinite Structures}
\author{Alexander Van Abel}
\affil{Wesleyan University}
\date{}

\maketitle

\begin{abstract}
We show that every definable subset of an uncountably categorical pseudofinite structure has pseudofinite cardinality which is polynomial (over the rationals) in the size of any strongly minimal subset, with the degree of the polynomial equal to the Morley rank of the subset. From this fact, we show that classes of finite structures whose ultraproducts all satisfy the same uncountably categorical theory are polynomial $R$-mecs as well as $N$-dimensional asymptotic classes, where $N$ is the Morley rank of the theory.
\end{abstract}

This article studies nonstandard cardinalities of definable sets in uncountably categorical pseudofinite structures. An $L$-structure $M$ is \emph{pseudofinite} if for every $L$-sentence $\varphi$, if $M \models \varphi$ then there is a finite $L$-structure $M_0$ such that $M_0 \models \varphi$. If $M$ satisfies this definition with the additional stipulation that $M_0$ is a substructure of $M$, then $M$ has the \emph{finite model property}.

Equivalently, a pseudofinite structure is one which is elementarily equivalent to an ultraproduct of finite structures. An ultraproduct of finite structures carries with it a notion of cardinality, which takes values in an ultrapower of the reals -- roughly, the pseudofinite cardinality of a definable set in an ultraproduct is the ultraproduct of the cardinalities of the ``slices'' of the definable set in the various finite structures. We make this definition precise in Section 2.

In Pillay's 2014 note ``Strongly minimal pseudofinite structures'' \cite{pillay}, he proves the following result:

\begin{fact}[\cite{pillay}, Theorem 1.1]
\label{pillay}
Let $D$ be a saturated pseudofinite strongly minimal structure, and let $q \in \mathbb{N}^\star$ be the pseudofinite cardinality of $D$ (written $q = |D|$). Then
\begin{enumerate}
\item for any definable (with parameters) set $X \subseteq D^n$, there is a polynomial $P_X(x)$ with integer coefficients and positive leading coefficient such that $|X| = P_X(q)$. Moreover $RM(X) = degree(P_X)$.

\item In fact, for any $L$-formula $\varphi(\bar{x},\bar{y})$, there are a finite number $P_1,\ldots,P_k$ of polynomials over $\mathbb{Z}$, and formulas $\psi_1(\bar{y}), \ldots, \psi_k(\bar{y})$, such that the $\psi_i(\bar{y})$ partition $\bar{y}$-space, and moreover for any $\bar{b}$, $|\varphi(\bar{x},\bar{b})(D)| = P_i(q)$ iff $\models \varphi_i(\bar{b})$.
\end{enumerate}
\end{fact}
 
At the end of his paper, Pillay makes this remark: ``It is also natural to ask what is the appropriate level of generality of the precise counting result in Theorem 1.1. Firstly there should be no problem obtaining a similar result for pseudofinite $\aleph_1$-categorical theories, where again any definable set will have cardinality an integral polynomial in $q$ where $q$ is the cardinality of a given strongly minimal set.''

In this paper, we give a proof of Pillay's suggested result, with the minor alteration that the polynomial has rational coefficients rather than just integers, in Theorem \ref{main}:

\begin{theorem*}
	Let $T$ be an uncountably categorical theory in the language $L$. Let $\theta(v,\bar{w})$ be an $L$-formula. Then for every $L$-formula $\varphi(x_1,\ldots,x_n,\bar{y})$, there are finitely many polynomials $F_1(X),\ldots,F_r(X) \in \mathbb{Q}[X]$ and $L$-formulas $\pi_1(\bar{y},\bar{w}),\\ \ldots,\pi_r(\bar{y},\bar{w})$ such that for all pseudofinite ultraproducts $M$ and all $\bar{d} \in M$ such that $D = \theta(M,\bar{d})$ is strongly minimal, we have that for all $\bar{b} \in M^{|\bar{y}|}$, the pseudofinite cardinality $|\varphi(M^n,\bar{b})|$ is $F_i(|D|)$ for some $i$, and furthermore for each $i$ the set \[\{\bar{b} \in M :  |\varphi(M^n,\bar{b})| = F_i(|D|)\}\] is definable over $X$ by $\pi_{\varphi, i}(\bar{y},\bar{d})$. 

	Additionally, if $\bar{b}$ satisfies $\pi_{\varphi, i}(\bar{y},\bar{d})$ , then the degree of the polynomial $F_i(X)$ is the Morley rank of the set $\varphi(M^n,\bar{b})$.
\end{theorem*}

We show that the stipulation that the coordinates are rational rather than integers is necessary in Example \ref{rationalex}. We observe that by letting $D$ be any strongly minimal subset of $M$, we obtain a direct analogue of Fact \ref{pillay}; we express our theorem in the stronger but more cumbersome form above in order to apply the theorem to the sequences of finite structures $(M_\lambda : \lambda \in \Lambda)$ for which ultraproducts $\prod_{\lambda \to \mathcal{U}} M_\lambda$ satisfy $T$.

Here we briefly explain our motivations for finding and proving this result. In the paper \cite{gms}, the authors demonstrate a number of results of the flavor that conditions on pseudofinite dimension in a pseudofinite ultraproduct (pseudofinite dimension is information derived from the pseudofinite cardinalities of the definable subsets of a structure) imply stability-theoretic properties of the stucture, such as simplicity and supersimplicity. In one such result, the authors show that if the pseudofinite dimension satisfies a property they refer to as ``strong attainability'' or ``(SA)'' for short, then the structure has a supersimple theory. The authors of that paper demonstrate that this result does not reverse, by providing an example of a pseudofinite ultraproduct with supersimple theory which does not satisfy (SA), although there is an elementarily equivalent pseudofinite ultraproduct which does. 

In \cite{vanabel}, we give an example of a supersimple pseudofinite theory $T$ such that no pseudofinite ultraproduct satisfying $T$ has the property (SA). As a general project, we are interested in finding converses to the the results in \cite{gms}, by which we mean finding conditions on $T$ which imply that conditions such as (SA) must hold, either in some pseudofinite ultraproduct or all pseudofinite ultraproducts satisfying $T$. In this paper, we obtain one such condition on $T$ -- uncountable categoricity -- as the conclusion of our Main Theorem implies the condition (SA).

 In addition to that paper, this work also connects with the notions of multidimensional exact class, as developed in \cite{wolf}, and $N$-dimensional asymptotic class as developed in \cite{elwes}. We prove that uncountably categorical pseudofinite theories give rise to both types of classes in Propositions \ref{rmec} and \ref{ndim}, summarized in the following theorem:
 
 \begin{theorem*}
 	Let $T$ be an uncountably categorical theory, and let $(M_\lambda : \lambda \in \Lambda)$ be a class of finite structures such that $M := \prod_\mathcal{U} M_\lambda \models T$ for any ultrafilter $\mathcal{U}$ on $\Lambda$. Then $(M_\lambda : \lambda \in \Lambda)$ is both a multidimensional exact class and an $N$-dimensional asymptotic class, where $N$ is the Morley rank of $M$.
 \end{theorem*}
 
The author thanks Alice Medvedev, Alf Dolich, Charlie Steinhorn, Alex Kruckman and Artem Chernikov for their conversations and suggestions for this paper.

\section{Notation}
Throughout this paper, $L$ denotes an arbitrary first-order language.

For a tuple $\bar{x}$, the expression $|\bar{x}|$ denotes the length of the tuple.

For an $L$-formula $\varphi(\bar{x},\bar{y})$, an $L$-structure $M$, a subset $X$ of $M^{|\bar{x}|}$ and a tuple $\bar{b} \in M$, the expression $\varphi(X,\bar{b})$ denotes the set $\{\bar{a} \in M^{|\bar{x}|} : \bar{a} \in X$ and $M \models \varphi(\bar{a},\bar{b})\}$.

Formulas $\varphi_1(\bar{v}),\ldots,\varphi_n(\bar{v})$ are said to \emph{partition $M^{|\bar{v}|}$} when the non-empty sets $\varphi_i(M^{|\bar{v}|})$ form a partition of $M^{|\bar{v}|}$ (we allow some of sets the $\varphi_i(M^{|\bar{v}|})$ to be empty).

The word ``rank'' in this paper will always refer to Morley rank. We will denote the Morley rank of a definable set $X$ by $MR(X)$.

\section{Pseudofinite Cardinality}

\begin{definition}
A theory $T$ with infinite models is \emph{pseudofinite} if every sentence $\sigma$ implied by $T$ has a finite model $M$ such that $M \models \sigma$.

Equivalently, $T$ is pseudofinite if there is some sequence of finite $L$-structures $(M_\lambda : \lambda \in \Lambda)$ and some ultrafilter $\mathcal{U}$ on $\Lambda$ such that the ultraproduct $\prod_{\lambda \to \mathcal{U}} M_i$ is a model of $T$.
\end{definition}

In this paper, the term ``pseudofinite ultraproduct'' will denote a model of the form $\prod_{\lambda \to \mathcal{U}} M_i$ for some sequence of finite $L$-structures $(M_\lambda : \lambda \in \Lambda)$ and some ultrafilter $\mathcal{U}$ on $\Lambda$

One nice property of ultraproducts of finite structures is that they come with a notion of subset ``cardinality''. Let $X$ be a definable subset of an ultraproduct $M = \prod_{\lambda \to \mathcal{U}} M_\lambda$ of finite structures. Then $X$ is the ultraproduct $\prod_{\lambda \to \mathcal{U}} X_\lambda$, where $X_\lambda \subseteq M_\lambda$ is defined by the same formula. Each $X_\lambda$ has a cardinality, being a finite set, and so we can say that $X$ has cardinality $(|X_\lambda|)_{\lambda \to \mathcal{U}}$, an element of the ultrapower $\mathbb{R}^\star = \prod_{\lambda \to \mathcal{U}} \mathbb{R}$. If $X$ is finite then $|X|$ will agree with the counting cardinality of $X$; if $X$ is infinite then $|X|$ will be an infinite hyperreal in $\mathbb{R}^\star$.

We can formalize this notion via the following construction, where we pass from our original language $L$ to a two-sorted expansion, where cardinality takes values in the second sort. We use the same formalism as in Section 2 of \cite{gms}.

\begin{definition}
\label{lplus}
Let $L$ be a first-order language. We define the expansion $L^+$ to be a two-sorted language. The home sort $\mathbf{H}$ has the language $L$. The second sort $\mathbf{OF}$ is an ordered field language $(0,1,-,+,\cdot,<)$. $L^+$ also has, for every $L$-formula $\varphi(\bar{x},\bar{y})$, a function symbol $f_{\varphi(\bar{x},\bar{y})} : \mathbf{H}^{|\bar{y}|} \to \mathbf{OF}$.
\end{definition}

\begin{definition}
Let $(M_\lambda : \lambda \in \Lambda)$ be a sequence of finite $L$-structures and let $\mathcal{U}$ be an ultrafilter on $\Lambda$. Let $M$ be the ultraproduct $\prod_{\lambda \to \mathcal{U}} M_\lambda$. We define the $L^+$-expansion $M^+$ by expanding each $M_\lambda$ to a $L^+$-structure $M_\lambda^+$. In $M_\lambda^+$, we define the ordered field sort $\mathbf{OF}$ to be the ordered field of real numbers $\mathbb{R}$. For an $L$-formula $\varphi(\bar{x},\bar{y})$ we define $f_{\varphi(\bar{x},\bar{y})}$ by $f_{\varphi(\bar{x},\bar{y})}(\bar{b}) = |\varphi(M_\lambda,\bar{b})|$, the cardinality of the set $\varphi(M_\lambda,\bar{b})$, for $\bar{b} \in M_\lambda^{|\bar{y}|}$.

Having defined each $M_\lambda^+$, we let $M^+$ be the ultraproduct of the sequence of $L^+$-structures $(M_\lambda^+ : \lambda \in \Lambda)$ with respect to $\mathcal{U}$.

For an $L$-formula $\varphi(\bar{x},\bar{y})$, we define the \emph{pseudofinite cardinality} of $\varphi(M,\bar{b})$ to be the hyperreal $f_{\varphi(\bar{x},\bar{y})}(\bar{b}) \in \mathbb{R}^\star$ in the $L^+$-expansion $M^+$. We denote this hyperreal as $|\varphi(M,\bar{b})|$.
\end{definition}

One benefit of this construction (not used in this paper, but used in e.g. \cite{gms}) is that the ultraproduct $M^+$ is $\omega_1$-saturated not just in $L$ but in the full language $L^+$. Later on in this paper we show that if $M$ is an uncountably categorical structure, then the two-sorted structure $M^+$ is no more complex than the multi-sorted disjoint union structure where one sort is $M$, the other sort is an ultrapower of the reals $\mathbb{R}^\star$, with no model-theoretic interaction between the two (Proposition \ref{disjointunion}).

We remark that the pseudofinite cardinality $|X|$ of a pseudofinite ultraproduct $M$ depends not just on the $L$-structure of $M$, but on the sequence of finite $L$-structures $(M_\lambda : \lambda \in \Lambda)$ and ultrafilter $\mathcal{U}$ on $\Lambda$ such that $M = \prod_{\lambda \to \mathcal{U}} M_\lambda$.

The following lemma is the combinatorial core of our main theorem. This lemma lets us express the cardinality of $X$ as a rational expression in terms of the cardinalities of subsets of a set $Y$ and cardinalities of the fibers of a definable relation $R(x,y)$ between $X$ and $Y$, so long as there are finitely many such cardinalities.

\begin{lemma}
\label{counting}

Let $\varphi(\bar{x},\bar{y},\bar{z})$ be a formula. Let $\bar{c} \in M^{|\bar{z}|}$ be parameters. Suppose that there are finitely many hyperreals $A^\star_1,\ldots,A^\star_n$ such that for all $\bar{b} \in M^{|\bar{y}|}$ there is an $i$ such that $\varphi(M^{|\bar{x}|},\bar{b},\bar{c}) = A^\star_i$. For each $i$, let $Z_i = \{\bar{b} \in M^{|\bar{y}|} : |\varphi(M^{|\bar{x}|},\bar{b},\bar{c})| = A^\star_i\}$ (note that $Z_i$ is definable in $L^+$).

\begin{enumerate}

\item The equation \[|\varphi(M^{|\bar{x}| + |\bar{y}|},\bar{c})| = \sum_{i=1}^n A^\star_i \cdot |Z_i|\] holds in $M^+$.

\item Suppose in addition that there is a single hyperreal $B^\star$ such that for all $\bar{a} \in M^{|\bar{x}|}$, if $\varphi(\bar{a},M^{|\bar{y}|}, \bar{c}) \neq \emptyset$ then $|\varphi(\bar{a},M^{|\bar{y}|}, \bar{c})| = B^\star$. Then the equation \[|\{\bar{a} \in M^{|\bar{x}|} : M \models \exists \bar{y} \varphi(\bar{a},\bar{y},\bar{c})\}| = \frac{ \sum_{i=1}^n A^\star_i \cdot |Z_i|}{B^\star}\] holds in $M^+$.

\end{enumerate}
\end{lemma}
\begin{proof}
The equation in Statement 1 is easily seen to hold in the finite case. Each $A^\star_i \cdot |Z_i|$ is the cardinality of the set $\{(\bar{a},\bar{b}) \in M^{|\bar{x}| + |\bar{y}|} : \bar{b} \in Z_i$ and $M \models \varphi(\bar{a},\bar{b})$\}.

To prove Statement 2, note that under our special assumption, Statement 1, with the roles of $\bar{x}$ and $\bar{y}$ switched, tells us that $|\varphi(M^{|\bar{x}| + |\bar{y}|},\bar{c})| = B^\star \cdot |\{\bar{a} \in M^{|\bar{x}|} : |\varphi(\bar{a},M^{|\bar{y}|}, \bar{c}) = B^\star\}| = B^\star \cdot |\{\bar{a} \in M^{|\bar{x}|} :  M \models \exists \bar{y} \varphi(\bar{a},\bar{y},\bar{c})\}|$. Therefore $\sum_{i=1}^n A^\star_i \cdot |Z_i| =  B^\star \cdot |\{\bar{a} \in M^{|\bar{x}|} L  M \models \exists \bar{y} \varphi(\bar{a},\bar{y},\bar{c})\}|$. Dividing both sides by $B^\star$ gives the desired result.
\end{proof}

We also make use of the following easily verified facts.

\begin{lemma}
\label{polynomeq}
Let $\mathbb{R}^\star = \prod_{\lambda \to \mathcal{U}} \mathbb{R}$ be an ultrapower of the real field $\mathbb{R}$.
\begin{enumerate}
\item Suppose $F(X), G(X) \in \mathbb{Q}[X]$ are polynomials, and suppose $A^\star \in \mathbb{R}^\star$ is an infinite positive element (i.e. $A^\star > n$ for every standard natural number $n$) such that $F(A^\star) = G(A^\star)$. Then $F(x) = G(x)$ for all $x \in \mathbb{R}$, i.e. $F$ and $G$ are the same polynomial.
\item Suppose $F(X) \in \mathbb{Q}[X]$ is a polynomial and suppose $A^\star, B^\star \in \mathbb{R}^\star$ are infinite positive elements such that $F(A^\star) = F(B^\star)$. Then $A^\star = B^\star$. 
\item Suppose $F(X) \in \mathbb{Q}[X]$ is a polynomial, and suppose $A^\star \in \mathbb{R}^\star$ is an infinite positive element such that $F(A^\star)$ is positive. Then the leading coefficient of $F(X)$ is positive.
\end{enumerate}
\end{lemma}
\begin{proof}
\begin{enumerate}
\item Express $A^\star$ as an ultralimit $(a_\lambda)_{\lambda \to \mathcal{U}}$, with each $a_\lambda \in \mathbb{R}$. Then $F(A^\star) = (F(a_\lambda))_{\lambda \to \mathcal{U}}$ and $G(A^\star) = (G(a_\lambda))_{\lambda \to \mathcal{U}}$. Since these two ultralimits are equal, we have $F(a_\lambda) = G(a_\lambda)$ for almost all $\lambda \in \Lambda$. Since $A^\star$ is infinite, for every standard $n \in \omega$ we have $a_\lambda > n$ for almost all $\lambda$. It follows that for every $n$ there is a real $a_\lambda > n$ such that $F(a_\lambda) = G(a_\lambda)$. Therefore the polynomial $F(x) - G(x)$ has arbitrarily large zeroes in $\mathbb{R}$. Hence it must be constant zero, and so $F(x) = G(x)$ for all $x$.
\item In $\mathbb{R}$, there is an $n$ such that $F(x)$ is strictly increasing or decreasing on $(n,\infty)$ (in $\mathbb{R})$). In particular, the function $x \mapsto F(x)$ is injective on this open ray. This is first-order expressible, hence true in $\mathbb{R}^\star$, and since $A^\star > n$ and $B^\star > n$ we get $A^\star = B^\star$.
\item If the leading coefficient of $F(x)$ were negative, then $F(a)$ is negative for sufficiently large real numbers $a$ -- that is, for $a > M$ for some real $M$. Therefore if $A^\star = (a_\lambda)_{\lambda \in \Lambda}$ is positive and infinite, the set $\{\lambda \in \Lambda : a_\lambda > M\}$ is $\mathcal{U}$-large. Then the set $\{\lambda \in \Lambda : F(a_\lambda) < 0\} = \{\lambda \in \Lambda : (F(A^\star))_\lambda\}$ is $\mathcal{U}$-large, whence $F(A^\star)$ is negative. Fact 3 follows contrapositively.
\end{enumerate}
\end{proof}

\section{Uncountably Categorical Theories}
We begin by recalling some basic facts about uncountably categorical theories.

\begin{fact}[\cite{buechler}, Lemma 3.4.10]
\label{satfact}
Let $M$ be an uncountable model of an $\aleph_1$-categorical theory $T$. Then $M$ is saturated.
\end{fact}

\begin{definition}
\label{deffinmor}
	Let $T$ be theory. We say that $T$ has \emph{definable and finite Morley Rank} if
	
	\begin{enumerate}
		\item Every definable subset of every model of $T$ has finite Morley rank, and
		
		\item For every formula $\varphi(\bar{x},\bar{y})$ and every $n < \omega$, there is a formula $Mor_{\varphi, n}(\bar{y})$ such that in any $M \models T$, $M \models Mor_{\varphi, n}(\bar{y})$ if and only if $MR(\varphi(M^{|\bar{x}|},\bar{b})) = n$
	\end{enumerate}
\end{definition}

\begin{fact}[\cite{pillaybook}, Chapter 1, Propositions 5.14 and 5.18]
\label{rankfact}
Let $T$ be an uncountably categorical theory. Then $T$ has definable and finite Morley rank.
\end{fact}

We use the following well-known fact about uncountably categorical models (see \cite[Lemma 3.1.12]{buechler}; this fact is also a consequence of Fact \ref{rankfact}, since a definable subset $X$ is finite if and only if $MR(X) = 0$).

\begin{fact}
\label{finitecover}
Suppose $T$ is $\aleph_1$-categorical and $\varphi(\bar{x},\bar{y})$ is a formula. Then there is a natural number $n$ such that in all models $M$ of $T$ and all $\bar{b} \in M^{|\bar{y}|}$, either $\varphi(M^{|\bar{x}|},\bar{b})$ is infinite or $|\varphi(M^{|\bar{x}|}, \bar{b})| \leq n$.
\end{fact}

\begin{definition}
\label{numt}
Let $T$ be an $\aleph_1$-categorical theory (or more generally any theory of which Fact \ref{finitecover} holds). Let $\varphi(\bar{x},\bar{y})$ be a formula. Define $Num_T(\varphi(\bar{x},\bar{y}))$ to be the least number $N \in \omega$ such that in any model $M$ of $T$ and any $\bar{b} \in M^{|\bar{y}|}$, if $\varphi(M^{|\bar{x}|},\bar{b})$ is finite then $|\varphi(M^{|\bar{x}|},\bar{b})| < N$.
\end{definition}

The following bit of folklore is well-known, although since the author is having a difficult time finding a citation, we provide a proof.

\begin{lemma}
\label{surj}
Let $M$ be an  $L$-structure. Let $A \subseteq M^n$ and $B \subseteq M^m$ be definable subsets such that $MR(A)$ and $MR(B)$ exist and are finite, and let $f : A \to B$ be a definable surjection such that $MR(f^{-1}(\bar{b}))$ exists and is finite for all $\bar{b} \in B$. Let $R, S \in \omega$. Then:
\begin{enumerate}
\item If $MR(f^{-1}(\bar{b})) \geq R$ for all $\bar{b} \in B$ and $MR(B) \geq S$, then $MR(A) \geq R + S$.

\item Suppose $Th(M)$ has finite and definable Morley rank (Definition \ref{deffinmor}) and suppose that $M$ is $\omega_1$-saturated. If $MR(f^{-1}(\bar{b})) \leq R$ for all $\bar{b} \in B$ and $MR(B) = S$ then $MR(A) \leq R + S$.
\end{enumerate}
\end{lemma}
\begin{proof}
	
\begin{enumerate}
	\item We prove this by induction on $S$. At $S = 0$, the statement is trivial. Specifically, let $\bar{b} \in B$ be any element. Since $f^{-1}(\bar{b}) \subseteq A$ we have $MR(A) \geq MR(f^{-1}(\bar{b})) \geq R = R + S$.
	
	Suppose the statement is proven for $s < S$ and suppose $MR(B) \geq S$. Let $B_1, B_2, \ldots$ be disjoint definable subsets of $B$ of rank $S-1$. For $i \in \omega$ let $A_i = f^{-1}(B_i)$. Then $A_1, A_2, \ldots,$ are disjoint definable subsets of $A$. Let $f_i$ be the restriction of $f$ to $A_i$. Then $f_i$ is a surjection of $A_i$ onto $B_i$. For $\bar{b} \in B_i$ we have $f_i^{-1}(\bar{b}) = f^{-1}(\bar{b})$ which has rank $\geq R$ by assumption. Therefore, by the inductive hypothesis, $MR(A_i) \geq R + S - 1$ for each $i$. Hence $MR(A) \geq R + S$.
	
	\item We prove Statement 2 by induction on $R + S$, with base case $S = 0$ (and $R$ equal to anything).
	
	If $S = 0$ then $MR(B) \leq S$ means $B$ is finite. Let $B = \{b_1,\ldots,b_p\}$. Then $A = f^{-1}(b_1) \cup \ldots \cup f^{-1}(b_p)$. If each $f^{-1}(b_i)$ has Morley rank $\leq R$, then as a finite union of such sets, $A$ also has a Morley rank of $\leq R = R + S$.
	
	Assume the statement is proven for all pairs $r,s$ such that $r + s < R + S$. Suppose first that the Morley degree of $B$ is 1, and suppose towards a contradiction that $MR(A) > R + S$. Then we can find $A_1, A_2, A_3, \ldots$ which are disjoint subsets of $A$ of Morley rank $R + S$. For each $i$, let $B_i = f(A_i)$. Then for each $i$, the map $f_i$ which is the restriction of $f$ to $A_i$ is a surjection from $A_i$ onto $B_i$. The fibers of $f_i$ are still of rank $\leq R$, as they are subsets of the fibers of $f$. If $MR(B_i) < S$ then by induction we obtain $MR(A_i) \leq R + MR(B_i) < R + S$, a contradiction. Therefore $MR(B_i) = S$ for each $i$, and since $B$ has degree 1, we obtain $MR(B \setminus B_i) < S$ for each $i$, as well as the fact that the Morley degree of each $B_i$ is 1.
	
	In each $B_i$, let $B_i^+ = \{\bar{b} \in B_i : MR(f_i^{-1}(\bar{b})) = R\}$ and let $B_i^- = \{\bar{b} \in B_i : MR(f_i^{-1}(\bar{b})) < R\} = B_i \setminus B_i^+$. Each $B_i^+$ and $B_i^-$ is definable, by definability of Morley rank in $Th(M)$. Let $A_i^+ = f_i^{-1}(B_i^+)$ and $A_i^- = f_i^{-1}(B_i^-) = A_i \setminus A_i^+$. Then $f_i$ is a surjection of $A_i^+$ onto $B_i^+$ and of $A_i^-$ onto $B_i^-$. If $R = 0$ then $B_i^-$ and $A_i^-$ are empty; if $R > 0$ then the fibers of $f_i$ restricted to $A_i^-$ have Morley rank $\leq R - 1$, so by induction, $MR(A_i^-) \leq R - 1 + S$. In either case, $MR(A_i^+) = MR(A_i) = R + S$ and $MR(A \setminus A_i^+) < R + S$. If $MR(B_i^+) < S$ then as in the previous paragraph, since $f_i$ surjects $A_i^+$ onto $B_i^+$ with fibers of rank $R$, we would have $MR(A_i^+) < R + S$, which is false. Therefore $MR(B_i^+) = S$ for each $i$, and since $B$ has degree 1, we have $MR(B \setminus B_i^+) < S$ for each $i$. Therefore for every $n$, the set $B - (B^+_1 \cap \ldots \cap B^+_n) = (B \setminus B_1^+) \cup \ldots \cup (B \setminus B_n^+)$ has Morley rank $< S$, and so $B^+_1 \cap \ldots \cap B^+_n$ has Morley rank $S$.
	
	Consider the partial type $\pi(\bar{y}) = \{``\bar{y} \in B^+_i" : i \in \omega\}$, which is definable with parameters from the set of parameters which define the sets $A_1, A_2, \ldots$. Since $B^+_1 \cap \ldots \cap B^+_n$ has Morley rank $S$ and is in particular non-empty for every $n$, the type $\pi$ is consistent. Because $M$ is $\omega_1$-saturated, $\pi$ is realized by some $\bar{b} \in B$. Then $MR(f_i^{-1}(\bar{b})) = R$ for all $i \in \omega$. But then $f_1^{-1}(\bar{b}), f_2^{-1}(\bar{b}), \ldots$ are disjoint subsets of $f^{-1}(\bar{b})$ of Morley rank $R$, hence $MR(f^{-1}(\bar{b})) > R$, a contradiction.
	
	This proves that $MR(A) \leq R + S$ when the Morley degree of $B$ is 1. If the degree of $B$ is $D > 1$, then $B$ is the disjoint union of definable sets $B_1, \ldots, B_D \subseteq B$, each of Morley rank $S$ and Morley degree 1. Letting $A_ i = f^{-1}(A_i)$ we get by the $D=1$ case that $MR(A_i) \leq R + S$, and since $A = A_1 \cup \ldots \cup A_i$ this completes the proof.
\end{enumerate}
\end{proof}

Our main tool in proving Theorem \ref{main} is an early result of Zil'ber's found in \cite{zilberpaper}, as well as his book \cite{zilberbook}. His result uses the concept of \emph{stratification}, defined as follows.

\begin{definition}[\cite{zilberbook}, Chapter 1, Section 2, after Fact 2.2]
 Let $A, B$ be definable unary subsets of a totally transcendental model $M$. A \emph{stratification of $A$ with respect to $B$} is a formula $\psi(x,v,\bar{c})$ with parameters $\bar{c}$ such that $A = \bigcup_{b \in B} \psi(M,b,\bar{c})$.
 
A \emph{stratification of rank at most $r$} is a stratification $\psi(x,v,\bar{c})$ such that $MR(\psi(M,b,\bar{c})) \leq r$ for all $b \in B$.

A \emph{proper stratification} is a stratification $\psi(x,y,\bar{c})$ of rank at most $rk(A) - 1$.
\end{definition}

We observe that for formulas $\varphi(x,\bar{y})$, $\theta(v,\bar{y})$ and $\psi(x,v,\bar{z})$, the notion \[\psi(x,v,\bar{c}) \mbox{ is a stratification of } \varphi(M,\bar{b}) \mbox{ with respect to } \theta(v,\bar{d})\] is first-order $\emptyset$-definable in $\bar{b},\bar{c},\bar{d}$. 

If we are working in a theory with definable and finite Morley Rank, then additionally, the notions ``$\psi(x,w,\bar{d})$ is a stratification of rank at most $r$ of $\varphi(M,\bar{b})$ with respect to $\theta(v,\bar{d})$'' and ``$\psi(x,v,\bar{c})$ is a proper stratification of $\varphi(M,\bar{b})$ with respect to $\theta(v,\bar{d})$'' are $\emptyset$-definable in $\bar{b},\bar{c},\bar{d}$.

\begin{definition}
\label{prstrat}
	Suppose $T$ is a theory with definable and finite Morley rank. Let $\varphi(x,\bar{y})$, $\psi(x,v,\bar{z})$, and $\theta(v,\bar{w})$ be $L$-formulas. Let $PrStrat_{T,\psi,\varphi,\theta}(\bar{y},\bar{z},\bar{w})$ be a formula such that in any $M \models T$, we have $M \models PrStrat_{T,\psi,\varphi,\theta}(\bar{b},\bar{c},\bar{d})$ if and only if $\psi(x,v,\bar{c})$ is a proper stratification of $\varphi(M,\bar{b})$ with respect to $\theta(M,\bar{d})$.
\end{definition}

In the sequel we will drop the subscripted $T$, as the theory will be clear from context.

Zil'ber's result tells us that in a $\aleph_1$-categorical theory, any infinite definable subset has a proper stratification over any strongly minimal set.

\begin{fact}[\cite{zilberpaper}, Lemma 1]
\label{stratfact}
Let $T$ be an uncountably categorical theory and let $M \models T$ be uncountable. Let $A, B \subset M$ be infinite definable sets, with $B$ strongly minimal. Then there exists a proper stratification of $A$ with respect to $B$.
\end{fact}

We need a slightly stronger statement for our main proof, which follows quickly from Zil'ber's result and from the saturation of any uncountable model of $T$.

\begin{corollary}
\label{stratcor}
Let $T$ be an uncountably categorical theory. Let $\varphi(x,\bar{y})$ and $\theta(v,\bar{w})$ be formulas. Then there are formulas $\psi_1(x,v,\bar{z}), \ldots, \psi_n(x,v,\bar{z})$ such that for all uncountable $M \models T$ and $\bar{b}, \bar{d} \in M$ such that $\theta(M,\bar{d})$ is strongly minimal, there is an $i$ such that for some $\bar{c} \in M$, the formula $\psi_i(x,v,\bar{c})$ is a proper stratification of $\varphi(M,\bar{b})$ with respect to $\theta(M,\bar{d})$.

Furthermore, we may choose formulas $\psi_i$ so that in the above paragraph, there is a \emph{unique} $i$ such that for some $\bar{c} \in M$, the formula $\psi_i(x,v,\bar{c})$ is a proper stratification.
\end{corollary}
\begin{proof}
Consider the partial type $\Pi(\bar{y},\bar{w})$ consisting of the formulas 

\[\forall \bar{z} [Mor_{\theta(x,\bar{w}) \wedge \pi(x,\bar{z}), 0}(\bar{z},\bar{w}) \vee Mor_{\theta(x,\bar{w}) \wedge \neg \pi(x,\bar{z}), 0}(\bar{z},\bar{w})]\]

and

\[ \forall{\bar{z}} [\neg PrStrat_{\psi,\varphi, \theta}(\bar{y},\bar{z},\bar{w})]\]

for all $L$-formulas $\pi(x,\bar{z})$ and $\psi(x,v,\bar{z})$ (recall $Mor_{\phi,n})$ from Definition \ref{deffinmor} and $PrStrat_{\varphi, \psi, \theta}$ from Defintion \ref{prstrat}.) If $M \models \Pi(\bar{b},\bar{d})$, the first formula schemata tells us that $\theta(M,\bar{d})$ is strongly minimal, and the second formula schemata tells us that there is no proper stratification of $\varphi(M,\bar{b})$ with respect to $\theta(M,\bar{d})$.

Suppose the first paragraph of the corollary were false. Then for any finite collection of $L$ formulas $\psi_1(x,v,\bar{z}), \ldots, \psi_n(x,v,\bar{z})$, there is an $M \models T$ and $\bar{b},\bar{d} \in M$ such that $\theta(M,\bar{d})$ is strongly minimal and such that for each $i$, there is no $\bar{c} \in M$ such that $\psi_i(x,v,\bar{c})$ is a proper stratification of $\varphi(M,\bar{b})$ with respect to $\theta(M,\bar{d})$. Then $\Pi(\bar{y},\bar{w})$ is consistent: we can realize any finite subset of $\Pi(\bar{y},\bar{w})$ with such a $\bar{b},\bar{d}$. Choose an uncountable $M$ which is a model of $T$. Then $M$ is a saturated model, by categoricity. Therefore $\Pi$ is realized by some $\bar{b},\bar{d} \in M$. Let $A = \varphi(M,\bar{b})$ and $B = \theta(M,\bar{d})$. Then $B$ is strongly minimal, and there is no proper stratification of $A$ with respect to $B$, contradicting Fact \ref{stratfact}.

Therefore the first paragraph of the corollary holds. We can force the index $i$ to be unique by modifying our formulas $\psi_i$ so that, for instance, $\psi_2(x,v,\bar{c})$ implies that $\psi_1(x,v,\bar{c})$ does not properly stratify $\varphi(M,\bar{b})$ with respect to $\theta(M,\bar{b'})$.
\end{proof}

The two statements in the following Lemma are noted and used in proofs by Zil'ber in \cite{zilberbook}. We use them to obtain the fact that the degree of the polynomial giving the cardinality of a definable set is exactly the Morley rank of the set. First we need another fact from \cite{zilberpaper} (also found in \cite{zilberbook}).
\begin{fact}[\cite{zilberpaper}, Theorem 3]
\label{stratrk}
Let $\Delta(x),\Phi(x)$ be $L$-formulas over $A \subset M$ with $MR(\Delta(M)) = m$ and let $\psi(v_0,v_1)$ be a stratification of $\Phi(M)$ with respect to $\Delta(M)$ of rank at most $n$. Then $MR(\Phi(M)) \leq m + n$.
\end{fact}

\begin{lemma}
\label{stratlem}
Let $M$ be a structure. Let $X, D \subseteq M$ be definable subsets with $D$ strongly minimal. Let $\psi(x,y)$ be a proper stratification of $X$ with respect to $D$ (possibly with parameters). 

\begin{enumerate}
\item There is an $N \in \omega$ such that $rk(\{a \in X : |\psi(a,D)| > N\}) < rk(X)$.

\item The set $\{d \in D : rk(\psi(X,d)) = rk(X) - 1\}$ is cofinite in $D$.
\end{enumerate}
\end{lemma}
\begin{proof}
\begin{enumerate}
\item As $D$ is strongly minimal, there is an $N \in \omega$ such that for all $a \in M$ we have that either $|\psi(a,D)| \leq N$ or $|\neg \psi(a,D)| \leq N$. Let $d_1, \ldots, d_N, d_{N+1} \in D$ be distinct elements. For each $i = 1, \ldots, N+1$ let $X_i = \psi(X,d_i)$. As $\psi$ is a proper stratification, we have $MR(X_i) < MR(X)$ for each $i$. Therefore $MR(X_1 \cup \ldots \cup X_{N+1}) < MR(X)$. If $a \in X$ and $|\psi(a,D)| > N$ then $\psi(a,D)$ is cofinite in $D$ with $|D - \psi(a,D)| \leq N$. Therefore at least one of $d_1,\ldots,d_{N+1}$ must be an element of $\psi(a,D)$, that is, $M \models \psi(a,d_i)$ for some $i$. Then $a \in \varphi(X,d_i) = X_i$. Hence $\{a \in X : |\psi(a,D)| > N\} \subseteq X_1 \cup \ldots \cup X_{N+1}$, and therefore $MR(\{a \in X : |\psi(a,D)| > N\}) \leq MR(X_1 \cup \ldots X_{N+1}) < MR(X)$.

\item Suppose otherwise. Then $\{d \in D : MR(\psi(M,d)) = MR(X) - 1\}$ is finite and $\{d \in D : MR(\psi(M,d)) \leq MR(X) - 2\}$ is cofinite. Let $\{d \in D : MR(\psi(M,d)) = MR(X) - 1\} = \{d_1,\ldots,d_n\}$. For $i = 1, \ldots, n$ let $X_i = \psi(M,d_i)$ and let $X' = \bigcup_{d \in D - \{d_1,\ldots,d_n\}} \psi(M,d)$. Then $X = X_1 \cup \ldots \cup X_n \cup X'$. Each $X_i$ has Morley rank $MR(X) - 1$. $X'$ is stratified by $\psi(x,y)$ over $D - \{d_1,\ldots,d_n\}$, and this stratification has rank at most $MR(X) - 2$. Therefore $MR(X') \leq MR(X) - 2 + MR(D - \{d_1,\ldots,d_n\}) = MR(X) - 1$. So $X$ is the union of finitely many sets of Morley rank at most $MR(X) - 1$, an impossibility. Therefore the set $\{d \in D : rk(\psi(X,d)) = rk(X) - 1\}$ must be cofinite in $D$.

\end{enumerate}

\end{proof}

\section{Main Theorem}
\begin{theorem}
\label{main}
Let $T$ be an uncountably categorical theory in the language $L$. Let $\theta(v,\bar{w})$ be an $L$-formula. Then for every $L$-formula $\varphi(x_1,\ldots,x_n,\bar{y})$, there are finitely many polynomials $F_1(X),\ldots,F_r(X) \in \mathbb{Q}[X]$ and $L$-formulas $\pi_1(\bar{y},\bar{w}),\ldots,\pi_r(\bar{y},\bar{w})$ such that for all pseudofinite ultraproducts $M$ and all $\bar{d} \in M$ such that $D = \theta(M,\bar{d})$ is strongly minimal, we have that for all $\bar{b} \in M^{|\bar{y}|}$, the pseudofinite cardinality $|\varphi(M^n,\bar{b})|$ is $F_i(|D|)$ for some $i$, and furthermore for each $i$ the set \[\{\bar{b} \in M :  |\varphi(M^n,\bar{b})| = F_i(|D|)\}\] is definable over $X$ by $\pi_{\varphi, i}(\bar{y},\bar{d})$. 

Additionally, if $\bar{b}$ satisfies $\pi_{\varphi, i}(\bar{y},\bar{d})$ , then the degree of the polynomial $F_i(X)$ is the Morley rank of the set $\varphi(M^n,\bar{b})$.
\end{theorem}

We first prove Theorem \ref{main} for the case $n = 1$ (Proposition \ref{onevarprop}), which takes up most of the proof. The full theorem will follow from a relatively fast inductive fiber-decomposition argument.

First we give an overview of our proof of the single-variable case, leaving aside the fine details to the forthcoming sequence of lemmas. Let $X \subseteq M$ be a definable subset of rank $R$ of which we will find the pseudofinite cardinality. Let $\psi(x,y)$ be a stratification of $X$ with respect to $D$, with parameters hidden for the moment. Partition $X$ into $X_1, X_2, \ldots, X_s$ so that the pseudofinite cardinality of the fiber $\psi(a,D)$ is a constant $A^\star_j$ over $a \in X_j$. By strong minimality of $D$, there are only finitely many such cardinalities. Then $|X| = \sum_{i=1}^s |X_i|$, so it sufices to show that each $X_i$ has cardinality $F(|D|)$ for some $F \in \mathbb{Q}[x]$ (the details of definability in this argument are in Lemma \ref{stratcase}). If $rk(X_j) < rk(X)$ then this is true by induction. If $rk(X_j) = rk(X)$ then this is shown in Lemma \ref{speciallem}. In that proof, we use Lemma \ref{counting} with respect to the relation $\psi(x,y)$ between $X_j$ and $D$ to express $|X_i|$ as a polynomial in the sizes of subsets of $D$ and fibers $\psi(X_i,d)$ for $d \in D$. Each of these cardinalities is itself a polynomial in $|D|$. For subsets of $D$ this holds by strong minimality, and for fibers $\psi(X_i,d)$ this holds by induction, as $\psi$ is a proper stratification.

We now prove Proposition \ref{onevarprop} rigorously, and in a particular form which lends itself to our inductive argument. 

\begin{proposition}
\label{onevarprop}
For every $R \in \omega$ and every $L$-formula $\varphi(x,\bar{y})$, there are polynomials $F_1(x),\ldots,F_n(x)$ of degree $R$ and $L$-formulas $\pi_1(\bar{y},\bar{w}),\ldots,$ $\pi_n(\bar{y},\bar{w})$ such that for all $\bar{b} \in M^{|\bar{y}|}$ and all $\bar{d} \in M^{|\bar{w}|}$ such that $D = \theta(M,\bar{d})$ is strongly minimal,

\begin{itemize}
\item If $MR(\varphi(x,\bar{b})) = R$ then $|\varphi(M,\bar{b})| = F_i(|D|)$ for some $i$, and
\item For all $i$, the set \[\{\bar{b} \in M^{|\bar{y}|} : MR(\varphi(M,\bar{b})) = R \mbox{ and } |\varphi(M,\bar{b})| = F_i(|D|)\}\] is definable by $\pi_i(\bar{y},\bar{d})$.
\end{itemize}
\end{proposition}

We will prove this proposition by induction on $R$, through a sequence of lemmas. First we handle the base case $R = 0$.

For this sequence of lemmas, recall (Definition \ref{numt}) that $Num_T(\varphi(\bar{x},\bar{y}))$ is the least number $N \in \omega$ such that whenever $M \models T$ and $\bar{b} \in M^{|\bar{y}|}$, if $\varphi(M^{|\bar{x}|},\bar{b})$ is finite then $|\varphi(M^{|\bar{x}|},\bar{b})| < N$.

\begin{lemma}
\label{baselem}
The statement of Proposition \ref{onevarprop} holds for $R = 0$.
\end{lemma}
\begin{proof}
Let $n = Num_T(\varphi(x,\bar{y}))$. Then for $M \models T$ and $\bar{b} \in M$ , if $RK(\varphi(M,\bar{b})) = 0$ -- which means $\varphi(M,\bar{b})$ is finite -- then $|\varphi(M,\bar{b})| < n$. We take our $F_1(x),\ldots,F_{n}(x)$ to be the constant degree-0 polynomials $0,\ldots,n-1$ and take each $\pi_i(\bar{y})$ to be an $L$-formula expressing ``$|\varphi(M,\bar{b})| = i-1$" -- note that this is a formula over $\emptyset$.
\end{proof}

Now assume $R > 0$ and that Proposition \ref{onevarprop} is proven for all ranks $r < R$.

Recall the formulas $Mor_{\varphi(x,\bar{y}), n}(\bar{y})$ from Definition \ref{deffinmor} and $PrStrat_{\psi,\varphi,\theta}(\bar{y},\bar{z},\bar{w})$ from Definition \ref{prstrat}.

\begin{lemma}
\label{speciallem}
Let $\varphi(x,\bar{y})$ be an $L$-formula. Suppose there is a formula $\psi(x,v,\bar{z})$ and a natural number $n$ such that for all $M \models T$, all $\bar{b} \in M^{|\bar{y}|}$, and all $\bar{d} \in M^{|\bar{z}|}$ such that $D_{\bar{d}} = \theta(M,\bar{d})$ is strongly minimal, if $MR(\varphi(M,\bar{b})) \leq R$ and $\varphi(M,\bar{b})$ is infinite then there is a $\bar{c} \in M^{|\bar{z}|}$ such that 

\begin{enumerate}
\item $\psi(x,v,\bar{c})$ is a proper stratification of $\varphi(M,\bar{b})$ with respect to $D_{\bar{d}}$
\item For every $a \in \varphi(M,\bar{b})$, the cardinality of the set $\psi(a,D_{\bar{d}},\bar{c})$ is $n$.
\end{enumerate}

Then Proposition \ref{onevarprop} holds of $R$ and $\varphi(x,\bar{y})$.
\end{lemma}

\begin{proof}

By the inductive hypothesis, there are polynomials $G_1,\ldots,G_m \in \mathbb{Q}[x]$ of degree $< R$ and formulas $\rho_1(v,\bar{z},\bar{w}),\ldots,\rho_m(v,\bar{z},\bar{w})$ such that for all models $M \models T$, all $\bar{d} \in M^{|\bar{w}|}$ such that $D_{\bar{d}} = \theta(M,\bar{d})$ is strongly minimal, and all $(e,\bar{c}) \in M^{1+|\bar{z}|}$, 

\begin{itemize}
\item if $MR(\psi(M,e,\bar{c})) < R$ then $|\psi(M,e,\bar{c})| = G_i(|D_{\bar{d}}|)$ for some $i$, and 
\item each $\rho_i$ defines the set $\{(e,\bar{c},\bar{d}) : MR(\psi(M,e,\bar{c})) < R$ and $|\psi(M,e,\bar{c})| = G_i(|D_{\bar{d}}|)\}$, and 
\item if $M \models \rho_i(e,\bar{c},\bar{d})$ then $\deg G_i = MR(\psi(M,e,\bar{c}))$.
\end{itemize}

Let $N$ be the maximum of the $2m$ numbers given by $Num_T(\xi_i(v;\bar{z}\bar{w}))$ and $Num_T(\xi'_i(v;\bar{z}))$ for $i = 1, \ldots, m$, where $\xi_i(v;\bar{z} \bar{w})$ is $\theta(v,\bar{w}) \wedge \rho_i(v,\bar{z},\bar{w})$ and $\xi'_i(v;\bar{z} \bar{w})$ is $\theta(v,\bar{w}) \wedge \neg \rho_i(v,\bar{z},\bar{w})$. Then if $\theta(M,\bar{d})$ is strongly minimal, we have that for each $i = 1, \ldots, m$ and for all $\bar{c}' \in M^{|\bar{z}|}$, either $|\rho_i(D,\bar{c}')| < N$ or $|D - \rho_i(D,\bar{c}')| < N$.

Let $\bar{b} \in Mor_{\varphi, R}(M^{|\bar{y}|})$ and $\bar{d} \in M^{|\bar{w}|}$ such that $D_{\bar{d}} = \theta(M,\bar{d})$ is strongly minimal. By the assumptions of our lemma, there is $\bar{c} \in M^{|\bar{z}|}$ such that $\psi(x,v,\bar{c})$ is a proper stratification of $\varphi(M,\bar{b})$ with respect to $D_{\bar{d}}$ -- i.e. $M \models PrStrat_{\psi,\varphi,\theta}(\bar{b},\bar{c},\bar{d})$ (this formula is defined in the remarks before Lemma \ref{speciallem}).

If $\bar{c}$ is such that $M \models PrStrat_{\psi,\varphi,\theta}(\bar{b},\bar{c},\bar{d})$ then since $\psi(x,w,\bar{c})$ is a proper stratification of $\varphi(M,\bar{b})$, the $x$-fibers $\psi(M,e,\bar{c})$ have rank $< R$ for all $e \in D$. Therefore for each $e \in D_{\bar{d}}$ there is a unique $i$ such that $M \models \rho_i(e,\bar{c},\bar{d})$. So the sets $\rho_1(D_{\bar{d}},\bar{c}), \ldots, \rho_m(D_{\bar{d}},\bar{c})$ partition $D_{\bar{d}}$. Therefore all but one of these sets is finite of cardinality $\leq N$, and one of these sets is cofinite of pseudofinite cardinality at least $|D_{\bar{d}}| - N$ (in fact, exactly equal to $|D_{\bar{d}}| - K$ where $K$ is the sum of the remaining finite cardinalities, which must therefore be no greater than $N$).

Let $\Sigma_{N,m}$ be the finite set $\{(\sigma_1,\ldots,\sigma_m) \in \{0,1,\ldots,N,\infty\}^m : \sigma_i = \infty$ for exactly one $i$, and $\sum_{j \neq i} \sigma_j \leq N\}$. 

For $\bar{b} \in Mor_{\varphi, R}(M^{|\bar{y}|})$ and $\bar{c} \in PrStrat_{\psi,\varphi,\theta}(\bar{b},M^{|\bar{z}|},\bar{d})$, let $\sigma(\bar{b},\bar{c},\bar{d})$ be the tuple $(\sigma_1,\ldots,\sigma_m) \in \Sigma_{N,m}$ such that $\sigma_i = |\rho_i(D_{\bar{d}},\bar{c},\bar{d})|$ if this set is finite, and $\sigma_i = \infty$ if $i$ is the unique index such that $\rho_i(D_{\bar{d}},\bar{c},\bar{d})$ is infinite. Note that for any $\sigma \in \Sigma_{N,m}$, the set $\{(\bar{b},\bar{c},\bar{d}) : \bar{b} \in Mor_{\varphi, R}(M^{|\bar{y}|})$ and $\bar{c} \in PrStrat_{\psi,\varphi,\theta}(\bar{b},M^{|\bar{z}|},\bar{d})$ and $\sigma(\bar{b},\bar{c},\bar{d}) = \sigma\}$ is definable by a formula $\pi_\sigma(\bar{y},\bar{z},\bar{w})$ over $\emptyset$.

Suppose $\sigma(\bar{b},\bar{c},\bar{d}) = \sigma = (\sigma_1,\ldots,\sigma_m)$. Let $i$ be the unique index such that $\sigma_i = \infty$, and let $n_\sigma := \sum_{j \neq i} \sigma_j$. Noting that for $j \neq i$ the set $\rho_j(D_{\bar{d}},\bar{c},\bar{d})$ is the set $\{b \in D_{\bar{d}} : |\psi(M,b,\bar{c})| = G_j(|D_{\bar{d}}|)\}$, we obtain by Lemma \ref{counting} the equation \[|\varphi(M,\bar{b})| = \frac{G_i(|D_{\bar{d}}|) \cdot (|D_{\bar{d}}| - n_\sigma) + \sum_{j \neq i} G_j(|D_{\bar{d}}|) \cdot \sigma_j}{n}.\] Therefore $|\varphi(M,\bar{b})| = F_\sigma(|D_{\bar{d}}|)$, where \[F_\sigma(X) = \frac{G_i(X) \cdot (X - n_\sigma) + \sum_{j \neq i} G_j(X) \cdot \sigma_j}{n} \in \mathbb{Q}[X].\]

By induction, each $G_j(X)$ has degree $< R$. Moreover, $G_i(X)$ has degree exactly $R - 1$, which follows from Lemma \ref{stratlem}.2, since $i$ is the unique $i$ such that $\{b \in D_{\bar{d}} : |\psi(M,b,\bar{c})| = G_i(D_{\bar{d}})\}$ is cofinite. Therefore $F_\sigma(X)$ is a polynomial of degree $R$.

Suppose $\bar{c}'$ is another element of $PrStrat_{\psi,\varphi,\theta}(\bar{b},M^{|\bar{z}|},\bar{d})$. Let $\sigma' = \sigma(\bar{b},\bar{c}',\bar{d})$ with coordinates $(\sigma_1',\ldots,\sigma_m')$. Let $i'$ be the unique index such that $\sigma'_{i'} = \infty$ and let $n_{\sigma'} = \sum_{j \neq i'} \sigma_j$. Then as before we have  \[|\varphi(M,\bar{b})| = \frac{G_{i'}(|D_{\bar{d}}|) \cdot (|D_{\bar{d}}| - n_{\sigma'}) + \sum_{j \neq i'} G_j(|D_{\bar{d}}|) \cdot \sigma'_j}{n}.\]

Therefore the polynomial $F_{\sigma'}(x) = \frac{G_{i'}(x) \cdot (x - n_{\sigma'}) + \sum_{j \neq i'} G_j(x) \cdot \sigma'_j}{n}$ agrees with the polynomial $F_\sigma(x)$ when applied to the infinite hyper-integer $|D_{\bar{d}}|$. By Lemma \ref{polynomeq}.1, the polynomials $F_\sigma$ and $F_{\sigma'}$ are equal.

This shows that the choice of $\bar{c} \in PrStrat_{\psi,\varphi,\theta}(\bar{b},M^{|\bar{z}|},\bar{d})$ does not affect the polynomial $F_\sigma$. Formally, we can quotient $\Sigma_{N,m}$ by the equivalence relation $\approx$ where $\sigma \approx \sigma'$ if $F_{\sigma} = F_{\sigma'}$. Then for each $\approx$-equivalence class $[\sigma]_{\approx}$, we can define the set $\{\bar{b} \in Mor_{\varphi, R}(M^{|\bar{y}|}) : \sigma(\bar{b},\bar{c},\bar{d}) \approx \sigma$ for all $\bar{c} \in PrStrat_{\psi,\varphi,\theta}(\bar{b},M^{|\bar{z}|},\bar{d})\}$ by $\pi_{[\sigma]_\approx}(\bar{y},\bar{d})$ for an $L$-formula $\pi_{[\sigma]_\approx}(\bar{y},\bar{w})$ (as $\Sigma_{N,m}$ is finite, we may explicitly list out the tuples which are equivalent to $\sigma$).  Then Proposition \ref{onevarprop} holds for $R$ and $\varphi(x,\bar{y})$, witnessed by the polynomials $F_\sigma(X)$ and formulas $\pi_{[\sigma]_\approx}(\bar{y},\bar{w})$ as $[\sigma]_\approx$ ranges over the finitely many equivalence classes.
\end{proof}

Now we can remove the restriction that the $v$-fibers of $\psi(x,v,\bar{z})$ have a constant finite cardinality.

\begin{lemma}
\label{stratcase}
Suppose there is a formula $\psi(x,v,\bar{w})$ such that for all $M \models T$, all $\bar{b} \in M^{|\bar{y}|}$ and all $\bar{d} \in M^{|\bar{w}|}$ such that $D_{\bar{d}} = \theta(M,\bar{d})$ is strongly minimal, if $MR(\varphi(x,\bar{b})) = R$ then there is a $\bar{c} \in M^{|\bar{w}|}$ such that $\psi(x,v,\bar{c})$ is a proper stratification of $\varphi(M,\bar{b})$ with respect to $D_{\bar{d}}$. Then Proposition \ref{onevarprop} holds for $\varphi(x,\bar{y})$.
\end{lemma}

\begin{proof}

Let $N$ be the maximum of $Num_T(\xi(v;x\bar{z}\bar{w}))$ and $Num_T(\xi'(v;x \bar{z}\bar{w}))$ (recall Definition \ref{numt}) where $\xi(v;x \bar{z}\bar{w})$ is $\theta(v,\bar{w}) \wedge \psi(x,v,\bar{z})$ and $\xi'(v;x \bar{z}\bar{w})$ is $\theta(v,\bar{w}) \wedge \neg \psi(x,v,\bar{z})$. Then if $D_{\bar{d}} := \theta(M,\bar{d})$ is strongly minimal, we have that for all $a',\bar{c}' \in M^{1+|\bar{z}|}$, either $|\psi(a,D_{\bar{d}},\bar{c}')| < N$ or $|D_{\bar{d}} - \psi(a,D_{\bar{d}},\bar{c}')| < N$.

For $1 \leq i \leq N-1$, let $\varphi_i(x,\bar{y},\bar{z},\bar{w})$ be the formula such that $\varphi_i(a,\bar{b},\bar{c},\bar{d})$ expresses ``$\varphi(a,\bar{b})$ and $ |\psi(a,D_{\bar{d}},\bar{c})| = i$''. Let $\varphi_0(x,\bar{y},\bar{z},\bar{w})$ be the formula such that $\varphi_0(a,\bar{b},\bar{c},\bar{d})$ expresses ``$\varphi_(a,\bar{b})$ and $|\psi(a,D_{\bar{d}},\bar{c})| > N$'', and note that $|\psi(a,D_{\bar{d}},\bar{c})| > N$ implies the set $\psi(a,D_{\bar{d}},\bar{c})$ is infinite, in fact cofinite, when $D_{\bar{d}}$ is strongly minimal. For $0 \leq i \leq N-1$ let $\psi_i(x,v,\bar{y},\bar{z},\bar{w})$ be $\psi(x,v,\bar{z}) \wedge \varphi_i(x,\bar{y},\bar{z},\bar{w})$.

Note that if $D_{\bar{d}}$ is strongly minimal and $\psi(x,w,\bar{c})$ is a stratification of $\varphi(M,\bar{b})$ with respect to $D_{\bar{d}}$, then the sets $\varphi_0(M,\bar{b},\bar{c},\bar{d}),\ldots,$ $ \varphi_N(M,\bar{b},\bar{c},\bar{d})$ partition $\varphi(M,\bar{b})$.

By induction and by Lemma \ref{speciallem}, for each $i$ there is an $s_i \in \omega$ and there are polynomials $G_{i,j}(X) \in \mathbb{Q}[X]$ and formulas $\pi_{i,j}(\bar{y},\bar{z},\bar{w})$ for $0 \leq i \leq N-1$ and $1 \leq j \leq s_i$ such that for each $i$ and all $\bar{b},\bar{c},\bar{d}$, if $D_{\bar{d}}$ is strongly minimal and either $MR(\varphi_i(M,\bar{b},\bar{c})) < R$ or $\varphi_i(M,\bar{b},\bar{c})$ is properly stratified by an instance of $\psi_i(x,v,\bar{y}\bar{z}\bar{w})$ with constant finite $D_{\bar{d}}$-fibers, then $|\varphi_i(M,\bar{b},\bar{c})| = G_{i,j}(|D_{\bar{d}}|)$ for some $j$, and for each $j$ the set of all such $(\bar{b},\bar{c}) \in M^{|\bar{y}| + |\bar{z}|}$ is defined by $\pi_{i,j}(\bar{y},\bar{z},\bar{d})$. To make notation easier, we can let all $s_i = s$ for some large enough $s$, by including additional (arbitrary) polynomials $G_{i,j}(X)$ and non-realized formulas $\pi_{i,j}(\bar{y},\bar{z},\bar{w})$ (for example, $y_1 \neq y_1)$ for $s_i < j \leq s$.

Suppose $D_{\bar{d}}$ is strongly minimal and $MR(\varphi(M,\bar{b})) = R$, and suppose $M \models PrStrat_{\psi,\varphi,\theta}(\bar{b},\bar{c},\bar{d})$, so that $\psi(x,v,\bar{c})$ is a stratification of $\varphi(M,\bar{b})$ with respect to $D_{\bar{d}}$. Then by Lemma \ref{stratlem}.1, $MR(\varphi_0(M,\bar{b},\bar{c},\bar{d})) < R$. If $MR(\varphi_i(M,\bar{b},\bar{c},\bar{d})) = R$, then \[\psi_i(x,w,\bar{b},\bar{c},\bar{d}) := \psi(x,w,\bar{c}) \wedge \varphi_i(x,\bar{b},\bar{c},\bar{d})\] is a proper stratification of $\varphi_i(M,\bar{b},\bar{c},\bar{d})$ with respect to $D_{\bar{d}}$, and the $w$-fibers $\psi_i(a,D_{\bar{d}},\bar{b},\bar{c},\bar{d})$ all have finite cardinality $i$, by the definition of $\varphi_i$. So in either case, the pseudofinite cardinality $|\varphi_i(M,\bar{b},\bar{c},\bar{d})|$ is $G_{i,j}(|D_{\bar{d}}|)$, where $j$ is the unique $j$ such that $M \models \pi_{i,j}(\bar{b},\bar{c},\bar{d})$.

Let $\Xi_{N,s}$ be the finite set of all tuples $\xi = (\xi_1,\ldots,\xi_N) \in \{1,\ldots,s\}^N$. For $\xi \in \Xi_{N,s}$, let $\pi_\xi(\bar{y},\bar{z},\bar{w})$ be the formula $\bigwedge_{0 \leq i \leq N} \pi_{i, \xi_i}(\bar{y},\bar{z},\bar{w})$. For $\bar{b} \in P_{\varphi,R}$ and $\bar{c} \in PrStrat_{\psi,\varphi,\theta}(\bar{b},M^{|\bar{z}|},\bar{d})$, let $\xi(\bar{b},\bar{c},\bar{d})$ be the tuple $\xi = (\xi_1,\ldots, \xi_N)$ such that $M \models \pi_{i,\xi(i)}(\bar{b},\bar{c},\bar{d})$ for each $i$ -- i.e., such that $M \models \pi_\xi(\bar{b},\bar{c},\bar{d})$.  Let $F_\xi(X) \in \mathbb{Q}[X]$ be the polynomial $\sum_{i=0}^N G_{i, \xi_i}(X)$.

When $D_{\bar{d}}$ is strongly minimal, the formulas $\{\pi_\xi(\bar{y},\bar{z},\bar{d}) : \xi \in \Xi_{N,s}\}$ partition \[\{(\bar{b},\bar{c}) \in M^{|\bar{y}|+|\bar{z}|} : MR(\varphi(M,\bar{b})) = R \mbox{ and } PrStrat_{\psi,\varphi,\theta}(\bar{b},\bar{c}, \bar{d})\}.\]    If $D_{\bar{d}}$ is strongly minimal and $M \models \pi_\xi(\bar{b},\bar{c},\bar{d})$ then the pseudofinite cardinality $|\varphi(M,\bar{b})|$ is $\sum_{i=0}^N |\varphi_i(M,\bar{b},\bar{c},\bar{d})|$ which equals $F_\xi(|D_{\bar{d}}|)$, since $|\varphi_i(M,\bar{b},\bar{c},\bar{d})| = G_{i,\xi(i)}(|D_{\bar{d}}|)$ for each $i$. We note that at least one $\varphi_i(M,\bar{b},\bar{c},\bar{d})$ must have full Morley rank (since this is a finite partition of $\varphi(M,\bar{b})$), and so at least one $G_{i,\xi(i)}(X)$ has degree $R$, hence $G_{\xi}(X)$ does as well.

As in the previous lemma, if $\bar{c}' \in M$ and $M \models PrStrat_{\psi,\varphi,\theta}(\bar{b},\bar{c}',\bar{d})$ and $\xi(\bar{b},\bar{c}',\bar{d}) = \xi'$, then as in the previous lemma we also have $|\varphi(M,\bar{b})| = F_{\xi'}(|D_{\bar{d}}|)$, so $F_{\xi}(x) = F_{\xi'}(x)$ as polynomials, again by Lemma \ref{polynomeq}.1. Then, as in that proof, we quotient $\Xi_{N,s}$ by the equivalence relation defined as $\xi \approx \xi'$ when $F_\xi = F_{\xi'}$, and taking our polynomials to be $F_{[\xi]_\approx}(X)$ and our defining formulas $\pi_{[\xi]_{\approx}}(\bar{y},\bar{w})$ to be ``for all $\bar{z}$, if $PrStrat_{\psi,\varphi,\theta}(\bar{y},\bar{z},\bar{w})$ holds then $\pi_{\xi'}(\bar{y},\bar{z},\bar{w})$ holds for some $\xi' \approx \xi$,'' which is definable since there are finitely many $\xi' \approx \xi$. 
\end{proof}

Now we can prove the full one-variable case.

\begin{proof}[Proof of Proposition \ref{onevarprop}]
By Corollary \ref{stratcor}, there are formulas\\ $\psi_1(x,v,\bar{z}),\ldots,\psi_s(x,v,\bar{z})$ such that for all $\bar{b} \in M^{|\bar{y}|}$ and all $\bar{d}$ such that $D_{\bar{d}}$ is strongly minimal, there is a unique $i \in \{1,\ldots,s\}$ such that for some $\bar{c} \in M^{|\bar{z}|}$, the formula $\psi_i(x,w,\bar{c})$ stratifies $\varphi(M,\bar{b})$ with respect to $D_{\bar{d}}$.  Apply Lemma \ref{stratcase} to each $\psi_i$ to obtain $s \in \omega$ and polynomials $G_{i,1}(X),\ldots,G_{i,s}(X) \in \mathbb{Q}[X]$ and formulas $\pi_{i,1}(\bar{y},\bar{w}),\ldots,\pi_{i,s}(\bar{y},\bar{w})$ for each $i$ as in the statement of Proposition \ref{onevarprop}. Taking our polynomials to be the polynomials $G_{i,j}(X)$ and our formulas to be ``$(\exists \bar{z} PrStrat_{\psi_i,\varphi,\theta}(\bar{y},\bar{z},\bar{w}))$ and $\pi_{i,j}(\bar{y},\bar{w})$'' proves the proposition.
\end{proof}

We can now the proof of the full main theorem.

\begin{proof}[Proof of Proposition \ref{main}]
The case $n = 1$ is proven in Proposition \ref{onevarprop}.

Let us assume the theorem is true for $n$ and let $\varphi(x_1,\ldots,x_n,x_{n+1},\bar{z})$ be a formula. By Proposition \ref{onevarprop}, there are $G_1(X),\ldots,G_s(X) \in \mathbb{Q}[X]$ of degree at most $N$ such that for all $\bar{a} \in M^n$, all $\bar{b} \in M^{|\bar{y}|}$ and all $\bar{d} \in M^{|\bar{w}|}$ such that $D_{\bar{d}}$ is strongly minimal, we have $|\varphi(M,\bar{a},\bar{b})| = G_i(|D_{\bar{d}}|)$ for some $i$, and for each $i$ the set of all $\bar{a}\bar{b}$ such that $|\varphi(M,\bar{a},\bar{b})| = G_i(|D_{\bar{d}}|)$ is definable by the formula $\pi_i(x_2,\ldots,x_{n+1},\bar{y},\bar{d})$. 

By induction, for each $i$ there are polynomials $H_{i,1}(X),\ldots,H_{i,m}(X) \in \mathbb{Q}[X]$ and formulas $\tau_{i,1}(\bar{y},\bar{w}), \ldots, \tau_{i,m}(\bar{y},\bar{w})$ such that for all $\bar{b} \in M^{|\bar{y}|}$ and $\bar{d} \in M^{|\bar{w}|}$ with $D_{\bar{d}}$ strongly minimal, there is a $j \in \{1,\ldots,m\}$ such that $|\pi_i(M^n,\bar{b},\bar{d})| = H_{i,j}(|D_{\bar{d}}|)$, and $deg H_{i,j}(X) = MR(\pi(\bar{x},\bar{b},\bar{d}))$, and for all $j$ the formula $\tau_{i,j}(\bar{y},\bar{d})$ defines the set of all $\bar{b}$ such that $|\pi_i(M^n,\bar{b},\bar{d})| = H_{i,j}(|D_{\bar{d}}|)$. (As in the proofs above, we may not have the same number of polynomials for each $\pi_i$, but we may add additional unused polynomials to get a uniform number).

For every tuple $\sigma = (\sigma_1,\ldots,\sigma_s) \in \{1,\ldots,m\}^s$, let $\tau_\sigma(\bar{z},\bar{w})$ be the formula $\bigwedge_i \tau_{i,\sigma_i}(\bar{z},\bar{w})$. Note that the formulas when $D_{\bar{d}}$ is strongly minimal the formulas $\tau_\sigma(\bar{z},\bar{d})$ partition $M^{|\bar{z}|}$ as $\sigma$ ranges over $\{1,\ldots,m\}^s$ (with some empty sets in this partition, if any of the polynomials $H_{i,j}$ is unattained). 

If $\sigma = (\sigma_1, \ldots, \sigma_s)$, $M \models \tau_\sigma(\bar{c})$, and $D_{\bar{d}}$ is strongly minimal, then \\ $|\pi_i(M^{|\bar{y}|}, \bar{b},\bar{d})| = H_{i,\sigma_i}(|D_{\bar{d}}|)$ for each $i$. By Lemma \ref{counting}.1, we have that \[|\varphi(M^{n+1},\bar{b})| = \sum_{i=1}^s F_i(|D_{\bar{d}}|) \cdot H_{i,\sigma_i}(|D_{\bar{d}}|) = F_\sigma(|D_{\bar{d}}|),\] where $F_\sigma(X)$ is the polynomial $\sum_{i=1}^s G_i(X) \cdot H_{i,\sigma_i}(X)$. We note that this equation comes from the fact that $\varphi(M^{n+1},\bar{b})$ is the disjoint union of the sets $Z_i := \{(c,\bar{a},\bar{b}) : M \models \pi_i(\bar{a},\bar{b})$ and $M \models \varphi(c,\bar{a},\bar{b})\}$ for $i = 1, \ldots, s$, each of which has cardinality $G_i(|D_{\bar{d}}|) \cdot H_{i, \sigma_i}(|D_{\bar{d}}|)$. The fibers $\{c \in M : (c,\bar{a}) \in Z_i\}$ all have pseudofinite cardinality $G_i(|D_{\bar{d}}|)$ by the definition of $Z_i$, hence they all have the same Morley rank which is $\deg G_i(X)$. Therefore by Lemma \ref{surj}, taking the definable surjection $f$ to be the projection of $Z_i$ onto the coordinates $(x_2,\ldots,x_n)$, the Morley rank of $Z_i$ is $\deg G_i + MR(\pi_i(M^n,\bar{b},\bar{d})) = \deg G_i(X) + \deg H_{i, \sigma_i}(X)$, which is the degree of the polynomial $G_i(X) \cdot H_{i,\sigma_i}(X)$. So $\deg F_\sigma(X) = \max (\deg G_i (X) \cdot H_{i, \sigma_i}(X) : i \leq s) = \max(MR(Z_i) : i \leq s) = MR(\varphi(M^{n+1},\bar{b}))$.
\end{proof}

\section{Additional Results}

We illustrate that our counting polynomials may need to have strictly rational coefficients, as opposed to the strongly minimal case, where the coefficients are integer-valued.

\begin{example}
\label{rationalex}
Let $L$ contain unary predicates $P_0,P_1$ and a binary relation $R$.

Let $T$ be the complete theory axiomatized by 

\begin{itemize}
\item $P_0$ and $P_1$ partition the universe $M$ into infinite disjoint sets

\item For all $x,y \in M$ if $R(x,y)$ holds then $x \in P_0$ and $y \in P_1$.

\item For all $x \in P_0$ there are exactly three $y \in P_1$ such that $R(x,y)$ holds.

\item For all $y \in P_1$ there are exactly two $x \in P_0$ such that $R(x,y)$ holds.

\item For all $x, x' \in P_0$ and $y, y' \in P_1$, if $R(x,y)$ and $R(x',y)$ and $R(x,y')$ hold then $R(x',y')$ holds.
\end{itemize}

The reduct of any model $M$ of $T$ to $\{R\}$ looks like an infinite disjoint union of complete two-to-three bipartite digraph, directed versions of the complete bipartite graph $K_{2,3}$, with arrows going from the side with two elements to the side with three elements. The predicate $P_0$ picks out the source nodes and the predicate $P_1$ picks out the target nodes. This theory is totally categorical. In any model $M$ of $T$, the sets $P_0(M)$ and $P_1(M)$ are strongly minimal. $T$ is a pseudofinite theory which may be satisfied by any ultraproduct of $M_1, M_2, M_3, \ldots$ where $M_n$ is the disjoint union of $n$ copies of the complete two-to-three bipartite digraph. Then $|M_n| = 5n$ and $|P_0(M_n)| = 2n$ and $|P_1(M_n)| = 3n$. It follows that in any infinite ultraproduct $M = \prod_{n \to \mathcal{U}} M_n$ we have $|P_0(M)| = \frac{2}{3} |P_1(M)|$, and $|P_1(M)| = \frac{3}{2} |P_0(M)|$. This shows that the polynomials in the theorem may be required to have strictly rational coefficients.

\end{example}

In our next result, we show that when a pseudofinite ultraproduct $M$ has uncountably categorical theory, then the two-sorted structure $M^+$ is uncomplicated in a model-theoretic sense. To formulate this result, we need the notion of a \emph{disjoint union} structure. Suppoe $L_1$ and $L_2$ are disjoint languages, and let $M_1$ be an $L_1$-structure and $M_2$ an $L_2$-structure. The disjoint union of $M_1$ and $M_2$ is a two-sorted structure. One sort is for $M_1$ and has the language $L_1$, and the other sort is for $M_2$ in the language $L_2$. Each sort inherits the full structure of the $M_i$, and there is no defined interaction between the two sorts. The disjoint union is, in a sense, the least model-theoretically complicated way of joining the structures $M_1$ and $M_2$. In this next proposition, we show that the counting structure $M^+$ is definable in the disjoint union of $M$ and the real closed field $\mathbb{R}^\star$.

\begin{proposition}
\label{disjointunion}
Let $(M_\lambda : \lambda \in \Lambda)$ be a sequence of finite $L$-structures and let $\mathcal{U}$ be an ultrafilter on $\Lambda$ such that $M := \prod_{\lambda \to \mathcal{U}} M_\lambda$ is $\aleph_1$-categorical. Let $M^+$ be the $L^+$-expansion of $M$ with respect to $(M_\lambda : \lambda \in \Lambda)$ and $\mathcal{U}$. Then $M^+$ is definable over a singleton $c \in \mathbf{OF}$ and a tuple $\bar{d} \in M$ in the two-sorted disjoint union structure of $M$ and $\mathbb{R}^\mathcal{U}$, where $\bar{d}$ is such that $\theta(M,\bar{d})$ is strongly minimal for an $L$-formula $\theta(v,\bar{w})$ and where $c$ is the pseudofinite cardinality of any infinite definable set $X \subseteq M^m$ (for any arity $m$).
\end{proposition}
\begin{proof}
Let $D = \theta(M,\bar{d})$ be a definable strongly minimal set. Then $|X| = F(|D|)$ for some polynomial $F(X) \in \mathbb{Q}[x]$, by Theorem \ref{main}. Since $|X|$ and $|D|$ are both non-standard integers, the hyperreal $|D|$ is the unique positive hyperreal $z$ such that $F(z) = |X|$, by Lemma \ref{polynomeq}.2. So $|D|$ is definable over $|X|$ in $\mathbb{R}^\mathcal{U}$.  Therefore to prove our theorem, it suffices to let $c = |D|$, and show that $M^+$ is definable in the disjoint union $L^+$-structure $M \cup \mathbb{R}^\mathcal{U}$ over $c  \in \mathbb{R}^\star$ and $\bar{d} \in M$.

To show that $M^+$ is definable in the disjoint union, we show that for every $L$-formula $\varphi(\bar{x},\bar{y})$, the function $f_{\varphi(\bar{x},\bar{y})}$ is definble over $c = |D| \in \mathbb{R}^\star$ and $\bar{d} \in M$. Let $\varphi(\bar{x},\bar{y})$ be an $L$-formula. By Theorem \ref{main}, there are finitely many polynomials $F_1(X),\ldots,F_r(X) \in \mathbb{Q}[X]$ and formulas $\pi_1(\bar{y},\bar{w}),\ldots,\pi_r(\bar{y},\bar{w})$ such that the formulas $\pi_i(\bar{y},\bar{d})$ partition $M^{|\bar{y}|}$ and for every $\bar{b}$ we have $|\varphi(M,\bar{b})| = F_i(|D|)$ iff $M \models \pi_i(\bar{b})$. Then for any $\bar{b} \in M^{|\bar{y}|}$ and any $e \in \mathbb{R}^\mathcal{U}$, \[M^+ \models f_{\varphi(\bar{x},\bar{y})}(\bar{b}) = e\] if and only if \[(M,\mathbb{R}^\mathcal{U}) \models \bigvee_{i=1}^r [\pi_i(\bar{b},\bar{d}) \wedge e = F_i(c)]\] which is definable in the disjoint union.
\end{proof}

In our final two propositions, we will demonstrate how Theorem \ref{main} passes down to give information about cardinalities of sets in finite structures. These propositions will be about particular kinds of families of sets defined as follows.

\begin{definition}
Let $L$ be a language and let $(M_\lambda : \lambda \in \Lambda)$ be a sequence of $L$-structures. We say that $(M_\lambda : \lambda \in \Lambda)$ has a \emph{zero-one law} if for every $L$-sentence $\varphi$, either $M_\lambda \models \varphi$ for all-but-finitely many $\lambda$ or $M_\lambda \models \neg \varphi$ for all-but-finitely many $\lambda$.

Equivalently, $(M_\lambda : \lambda \in \Lambda)$ has a zero-one law if any two non-principal ultraproducts of the family are elementarily equivalent.

We also say that $(M_\lambda : \lambda \in \Lambda)$ is a \emph{zero-one class}.

If $(M_\lambda : \lambda \in \Lambda)$ is a zero-one clas, we call the theory of any/all non-principal ultraproducts of the family the \emph{limit theory}.
\end{definition}

Let us also recall the little-o notation ``$f = o(g)$'', which is an abbreviation for the statement \[\lim_{x \to \infty} \frac{f(x)}{g(x)} = 0.\]

First we show that a zero-one class with an uncountably categorical limit theory is an $R$-mec, a notion devised by Anscombe, Macpherson, Steinhorn and Wolf to appear in their upcoming paper \cite{amsw}, and explored in detail in Wolf's thesis \cite{wolf}.

\begin{definition}
\label{rmecdef}
Let $\mathcal{C}$ be a class of finite $L$-structures. Let $R$ be a set of functions from $\mathcal{C}$ to $\mathbb{R}^{\geq 0}$. Then $\mathcal{C}$ is an \emph{$R$-mec} if for every $L$-formula $\varphi(\bar{x},\bar{y})$ there are finitely many $h_1(X),\ldots,h_n(X) \in R$ and $L$-formulas $\psi_1(\bar{y}),\ldots,\psi_n(\bar{y})$ such that for each $M \in \mathcal{C}$,

 \begin{itemize}
 \item for each $\bar{b} \in M^{|\bar{y}|}$, there is an $i$ such that $|\varphi(M,\bar{b})| = h_i(M)$, and
 
 \item for each $i$ the formula $\psi_i(\bar{y})$ defines the set $\{\bar{b} \in M^{|\bar{y}|} : |\varphi(M,\bar{b})| = h_i(M)\}$.
 \end{itemize}
 \end{definition}

The word ``mec'' is short for ``multidimensional exact class''. It is a special case of the more general notion of multidimensional asymptotic class, introduced in \cite{amsw} and explored in \cite{wolf} as $R$-mecs are, the definition of which is similar to the above definition except instead of stipulating that $|\varphi(M,\bar{b})| = h_i(M)$, we stipulate that the two are ``asymptotically'' equal, in the sense that the quotient $\frac{|\varphi(M,\bar{b})| - h_i(M)}{h_i(M)}$ tends to zero as $h_i(M)$ goes to infinity. That is, $|\varphi(M,\bar{b})| - h_i(M) = o(h_i(M))$, as $M$ ranges over $\mathcal{C}$ and $\bar{b}$ ranges over $\psi_i(M^{\bar{y}|})$. We note that $h_i$ takes as input the structure $M$, not the number $|M|$.

In this proposition, we prove that a class of finite structures whose ultraproduct theory is uncountably categorical is an $R$-mec for a particularly simple class of functions $R$.

\begin{proposition}
\label{rmec}
Let $T$ be an uncountably categorical pseudofinite theory. Suppose $(M_\lambda : \lambda \in \Lambda)$ is a sequence of finite $L$-structures with a zero-one law and limit theory $T$. Let $\theta(v,\bar{w})$ be a formula such that $\theta(M,\bar{d})$ is strongly minimal for some $M \models T$ and $\bar{d} \in M$. Then for every formula $\varphi(\bar{x},\bar{y})$, there are polynomials $F_1,\ldots,F_r \in \mathbb{Q}[x]$ and formulas $\pi_1(\bar{y},\bar{w}),\ldots,\pi_r(\bar{y},\bar{w})$ such that in all finite structures $M_\lambda$ there is a tuple $\bar{d}_\lambda \in M_\lambda^{|\bar{w}|}$ such that for all $\bar{b} \in M_\lambda^{|\bar{y}|}$, there is $i \in \{1,\ldots,r\}$ such that $|\varphi(M_\lambda^{|\bar{x}|}, \bar{b})| = F_i(|\theta(M_\lambda,\bar{d}_\lambda)|)$, and for each $i$ the set $\pi_i(M_\lambda^{|\bar{y}|},\bar{d}_\lambda)$ is the set of all $\bar{b} \in M_n^{|\bar{y}|}$ for which this equation holds.

In particular, the class $(M_\lambda : \lambda \in \Lambda)$ is an $R$-mec, where $R$ is the set of functions defined by \[f(M_\lambda) = F(|\theta(M_\lambda,\bar{d}_\lambda)|)\] as $F(x)$ ranges over $\mathbb{Q}[x]$ and $\bar{d}_\lambda$ ranges over elements of $M_\lambda$.
\end{proposition}

\begin{proof}
Let $\theta(v,\bar{w})$ be an $L$-formula such that some/every (saturated) model $M$ of the theory $T$ has a $\bar{d} \in M$ such that $\theta(M,\bar{d})$ is strongly minimal. 

Let $\varphi(\bar{x},\bar{y})$ be an $L$-formula. Let $F_1(X),\ldots,F_r(X) \in \mathbb{Q}[X]$ be the polynomials and $\pi_1(\bar{y},\bar{w}),\ldots,\pi_r(\bar{y},\bar{w})$ be the formulas obtained in Theorem \ref{main} as applied to $\varphi(\bar{x},\bar{y})$ and $\theta(v,\bar{w})$. 

Let $\Psi(\bar{w})$ be an $L^+$-formula such that $\Psi(\bar{d})$ expresses the conjunction of the two conditions

\begin{enumerate}

\item $\pi_1(\bar{y},\bar{d}),\ldots,\pi_r(\bar{y},\bar{d})$ partition $M^{|\bar{y}|}$, and

\item for all $\bar{b} \in M^{|\bar{y}|}$ we have $|\varphi(M^{|\bar{x}|},\bar{b})| = F_i(|\theta(M,\bar{d})|)$ if and only if $M \models \pi_i(\bar{b},\bar{d})$.

\end{enumerate}

Then for every $M \models T$ and $\bar{d} \in M$, if $\theta(M,\bar{d})$ is strongly minimal then $M^+ \models \Psi(\bar{d})$. Since every uncountable model of $T$ is saturated, every uncountable model of $T$ contains a tuple $\bar{d}$ such that $\theta(M,\bar{d})$ is strongly minimal. Therefore whenever $M$ is a pseudofinite ultraproduct which satisfies $T$, the expanded structure $M^+$ satisfies the $L^+$-sentence $\exists \bar{w} \Psi(\bar{w})$. 

By assumption, every nonprincipal ultraproduct of the family $(M_\lambda : \lambda \in \Lambda)$ satisfies the theory $T$. Therefore every infinite ultraproduct of the expanded finite structures $M^+_\lambda$ satisfies the sentence $\exists \bar{w} \Psi(\bar{w})$. Hence this sentence is satisfied in all but finitely many of the $L^+$-expansions $M^+_\lambda$. If $M^+_\lambda$ satisfies $\exists \bar{w} \Psi(\bar{w})$ then, taking $\bar{d}_\lambda$ to be a witness, the structure $M^+_\lambda$ satisfies the conclusion of the proposition.

Let $M_1, \ldots, M_k$ be the finitely many structures in our family whose expansions $M_i^+$ do not satisfy the $L^+$-formula $\exists \bar{w} \Psi(\bar{w})$. Let $N$ be a number greater than $\max \{ |M_1|, \ldots, |M_k|\}$. Let us add polynomials $F_{r+1}(X),\ldots,F_{r+N}(X)$ where $F_{r+k}(X)$ is the constant polynomial $k-1$, and add new formulas $\pi_{r+1}(\bar{y},\bar{w}),\ldots,$ $\pi_{r+N}(\bar{y},\bar{w})$ where $\pi_{r+k}(\bar{b},\bar{d})$ expresses ``$|\varphi(M,\bar{b})| = k-1$''. Finally, let us modify the formulas $\pi_1(\bar{y},\bar{w}), \ldots, \pi_r(\bar{y},\bar{w})$ by conjoining them each with a formula such that $\pi_i(\bar{b},\bar{d})$ implies ``$|\varphi(M,\bar{b})| \geq N$''. With this alteration, the conclusion of the proposition holds for every structure $M_\lambda$.
\end{proof}

The following notion was introduced by Elwes in \cite{elwes}. 

\begin{definition}
\label{ndimdef}
A family of finite $L$-structures $(M_\lambda : \lambda \in \Lambda)$ is a \emph{N-dimensional asymptotic class} if for every $L$-formula $\varphi(\bar{x},\bar{y})$, there exist finitely many pairs $(\mu_1,d_1),\ldots,(\mu_s,d_s) \in \mathbb{R}^{\geq 0} \times \omega$ and $L$-formulas $\pi_1(\bar{y}),\ldots\pi_s(\bar{y})$ over $\emptyset$ such that for every $\lambda \in M_\lambda$, the following two conditions hold:

\begin{itemize}
\item The formulas $\pi_1(\bar{y}),\ldots,\pi_s(\bar{y})$ partition $M_\lambda^{|\bar{y}|}$,

\item For $i = 1, \ldots, s$, we have that $|\varphi(M_\lambda^{|\bar{x}|}, \bar{b})| - \mu_i |M_\lambda|^{d_i / N} = o(|M_\lambda|^{d_i/N})$ as $|M_\lambda| \to \infty$, for $\bar{b} \in \pi_i(M_\lambda^{|\bar{y}|})$.
\end{itemize}

Clause 2 may be restated as: for every $\epsilon > 0$ there is a number $C$ such that whenever $|M_\lambda| > C$ and $\bar{b} \in \pi_i(M_\lambda^{|\bar{y}|})$ we have \[\big{|}|\varphi(M_\lambda^{|\bar{x}|}, \bar{b}) - \mu_i |M_\lambda|^{d_i / N}\big{|} < \epsilon |M_\lambda|^{d_i / N}.\]

\end{definition}

We prove that a zero-one family of finite structures whose limit theory is uncountably categorical is an $N$-dimensional asymptotic class, where $N$ is the Morley rank of the limit theory. As the definition of an $N$-dimensional asymptotic class requires that the formulas $\pi_i$ be definable without parameters, while the conclusion of Proposition \ref{rmec} allows for parameters, we shall need the following fact, which for example may be found as a consequence of \cite[6.1.16]{marker}.

\begin{fact}
\label{isosm}
Let $T$ be an uncountably categorical theory. Then every model of $T$ contains a strongly minimal set defined by a formula whose parameters have an isolated type. That is, there exist formulas $\kappa(\bar{w})$ and $\theta(v,\bar{w})$ where $\kappa(\bar{w})$ isolates a type $p$ such that for a model $M \models T$ and $\bar{d} \in M$, if $tp(\bar{d}) = p$ then $\theta(M,\bar{d})$ is strongly minimal.
\end{fact}

We shall need the following corollary of Proposition \ref{rmec}.

\begin{corollary}
\label{rmeccor}
Assume that $(M_\lambda : \lambda \in \Lambda)$ is a zero-one class with uncountably categorical limit theory $T$. Let $\kappa(\bar{w})$ and $\theta(v,\bar{w})$ be as in Fact \ref{isosm}. Then in the conclusion  of Proposition \ref{rmec}, the tuple $\bar{d}_\lambda$ may be taken to be any tuple in $\kappa(M_{\lambda}^{|\bar{w}|})$, provided that this set is nonempty, which is the case in cofinitely many $M_\lambda$.
\end{corollary}

\begin{proof}
We slightly modify the proof of Proposition \ref{rmec}. We conjoin to $\Psi(\bar{w})$ the formula $\kappa(\bar{w})$. Then it remains true that all-but-finitely many of the finite expansions $M^{+}_\lambda$ satisfy $\exists \bar{w} \Psi(\bar{w})$, and we may explicitly encode out the finitely many counterexamples as in the proof of Proposition \ref{rmec}.
\end{proof}

For the counting clause of the definition of $N$-dimensional asymptotic class, we shall apply the following lemma about the asymptotics of the inverse function of a polynomial.

\begin{lemma}
\label{polynom}
Let $F(X) = a_n x^n + a_{n-1} x^{n-1} ... + a_1 x + a_0 \in \mathbb{R}[X]$, with $a_n > 0$. Let $C \in \mathbb{R}$ be such that the function $F(x)$ is increasing on $[C,\infty)$, and let $F^{-1} : [F(C),\infty) \to [C,\infty)$ be the inverse function. Then \[\lim_{x \to \infty} [(x / a_n)^{1/n} - F^{-1}(x)] = \frac{a_{n-1}}{n a_n}.\]

Furthermore, let $G(x) = b_m x^m + b_{m-1} x^{m-1} + \ldots + b_1 x + b_0$. Then \[(b_m/ a_n^{m/n}) x^{m/n} - G(F^{-1}(x)) = o(x^{m/n}).\]
\end{lemma}

\begin{proof}
We compare $F(X)$ to polynomials of the form \[a_n(x + \alpha)^n = a_n x^n + a_n \cdot n \alpha x^{n-1} + \ldots + a_n \cdot n \alpha^{n-1} x + a_n \alpha^n.\] Note that the inverse of $a_n(x + \alpha)^n$ is $(x / a_n)^{1/n} - \alpha$.

It is easily seen that $a_n(x+\alpha)^n > F(x)$ for sufficiently large $x$ if $a_n \cdot n \alpha > a_{n-1}$ -- that is, if $\alpha > a_{n-1} / (n a_n)$ -- and that $a_n(x+\beta)^n < F(x)$ for sufficiently large $x$ if $\beta < a_{n-1} / (n a_n)$. 

In general, if $f(x),g(x) :\mathbb{R} \to \mathbb{R}$ are increasing functions such that $f(x) > g(x)$ for all (sufficiently large) $x$, then $f^{-1}(x) < g^{-1}(x)$ for all (sufficiently large) $x$. Therefore if $\alpha > a_{n-1} / (n a_n)$ we have that $(x/a_n)^{1/n} - \alpha < F^{-1}(x)$ eventually, and if $\beta <  a_{n-1} / (n a_n)$ then $(x/a_n)^{1/n} - \beta > F^{-1}(x)$ eventually. Rearranging these inequalities, we obtain that $\beta < (x/a_n)^{1/n} - F^{-1}(x) < \alpha$ eventually for all $\alpha >  a_{n-1} / (n a_n) > \beta$, which proves that \[\lim_{x \to \infty} (x/a_n)^{1/n} - F^{-1}(x) = a_{n-1} / (n a_n).\]

To prove that \[(b_m/ a_n^{m/n}) x^{m/n} - G(F^{-1}(x)) = o(x^{m/n})\] it suffices to prove the asymptotic equations \[b_m(x/a_n)^{m/n} - G((x / a_n)^{1/n}) = o(x^{m/n})\] and \[G((x / a_n)^{1/n}) - G(F^{-1}(x))= o(x^{m/n}).\] For the first, rewrite \[\lim_{x \to \infty} \frac{b_m(x/a_n)^{m/n} - G((x / a_n)^{1/n})}{x^{m/n}}\] as \[\lim_{z \to \infty} \frac{b_m z^m - G(z)}{a_n^{m/n} z^m}\] after making the variable substitution $z = (x / a_n)^{1/n}$. The numerator is a polynomial of degree less than $m$, therefore the limit is 0 as required.

To verify the second asymptotic equation, we observe that the first part of this lemma implies the existence of a $B$ such that $F^{-1}(x) \in ((x/a_n)^{1/n} - B, (x/ a_n)^{1/n} + B)$ for all $x$. Then\[|G(F^{-1}(x)) - G((x/a_n)^{1/n})| < |G((x/a_n)^{1/n} + B) - G((x/a_n)^{1/n} - B)\] for sufficiently large $x$. Therefore it suffices to show \[\lim_{x \to \infty} \frac{{G((x/a_n)^{1/n} + B) - G((x/a_n)^{1/n} - B)}}{x^{m/n}} = 0\] which is equivalent to saying \[\lim_{z \to \infty} \frac{G(z + B) - G(z - B)}{z^m} = 0\] as in the previous calculation. The polynomial $G(z + B) - G(z - B)$ simplifies to a polynomial of degree less than $m$, and so the limit is correct, and the conclusion of the lemma is true.
\end{proof}

With this lemma, we are able to prove that any class of finite structures with a zero-one law (as defined earlier) and uncountably categorical limit theory is an $N$-dimensional asymptotic class, where $N$ is the Morley rank of the limit theory.

\begin{proposition}
\label{ndim}
Let $T$ be an uncountably categorical pseudofinite theory. Suppose $(M_\lambda : \lambda \in \Lambda)$ is a sequence of finite $L$-structures such that $\prod_{\lambda \to \mathcal{U}} M_\lambda \models T$ for every non-principal ultrafilter $\mathcal{U}$ on $\omega$. Then $(M_\lambda : \lambda \in \Lambda)$ is an $N$-dimensional asymptotic class, where $N$ is the Morley rank of $T$.
\end{proposition}

\begin{proof}
 Let $\kappa(\bar{w})$, $\theta(v,\bar{w})$ and $p$ be as in Fact \ref{isosm}. 
 
First we show Definition \ref{ndimdef} holds for the formula $\theta(x,\bar{y}) \wedge \kappa(\bar{w})$ (with variables as in that definition partitioned as $x,\bar{y}\bar{w})$. We apply Corollary \ref{rmeccor} to the formula ``$x = x$'' to obtain polynomials \[F_1(X),\ldots,F_r(X) \in \mathbb{Q}[X]\] of the Morley rank of the formula (which is $N$) and formulas $\pi_1(\bar{w}),\ldots,\pi_r(\bar{w})$. Let $i$ be such that $p \vdash \pi_i(\bar{w})$. Then $M_\lambda \models (\kappa(\bar{w}) \rightarrow \pi_i(\bar{w}))$ for  cofinitely many $M_\lambda$. For such a $M_\lambda$, if $M_\lambda \models \kappa(\bar{d})$ then $|M_\lambda| = F_i(|\theta(M_\lambda,\bar{d})|)$. If $M_\lambda \models \kappa(\bar{d}')$ as well then $|M_\lambda| = F_i(|\theta(M_\lambda, \bar{d}')|)$. By injectivity of polynomials on a tail, this implies the existence of a $C$ such that for $|M_\lambda| > C$, if $\bar{d}, \bar{d}' \in \kappa(M_\lambda^{|\bar{w}|})$ then $|\theta(M_\lambda, \bar{d})| = |\theta(M_\lambda, \bar{d}')|$, and $|M_\lambda| = F_i(|\theta(M_\lambda, \bar{d})|)$. Let $F(X) = F_i(X) = a_N X^N + \ldots + a_1 X + a_0$. It follows from Lemma \ref{polynom} that \[\lim_{|M_\lambda| \to \infty} \frac{|\theta(M_\lambda, \bar{d})| - (1/a_N)^N |M_\lambda|^{1/N}}{|M_\lambda|^{1/N}} = 0\] if we take $\bar{d} \in \kappa(M_\lambda^{|\bar{w}|})$.

Now let $\varphi(x_1,\ldots,x_n,\bar{y})$ be an arbitrary $L$-formula. Apply Corollary \ref{rmeccor} to $\varphi(x_1,\ldots,x_n,\bar{y})$ to get polynomials $G_1(X),\ldots,G_s(X)$ and formulas $\pi_1(\bar{y},\bar{w}), \ldots,$ $ \pi_s(\bar{y},\bar{w})$ such that the conclusion of that proposition holds. Then if $\bar{b},\bar{d} \in M_\lambda$ with $|M_\lambda|$ sufficiently large and $M_\lambda \models \kappa(\bar{d}) \wedge \pi_i(\bar{b},\bar{d})$ then $|\varphi(M_\lambda^n,\bar{b})| = G_i(|\theta(M_\lambda,\bar{d})|) = G_i(F^{-1}(|M_\lambda|))$. For each $i = 1,\ldots,s$ let $G_i(X) = a_{i,N_i} X^{N_i} + \ldots + a_{i,1} X + a_{i,0}$. By Lemma \ref{polynom}, we have \[\lim_{x \to \infty} \frac{G_i(F^{-1}(x)) - (a_{i,N_i} / a_N^{1/N}) x^{N_i / N}}{x^{N_i / N}} = 0.\] Therefore we have \[\lim_{x \to \infty} \frac{|\varphi(M_\lambda^n, \bar{b})| - (a_{i,N_i} / a_N^{1/N}) |M_\lambda|^{N_i / N}}{|M_\lambda|^{N_i / N}} = 0\] for all $\bar{b},\bar{d} \in M_\lambda$ with $M_\lambda \models \pi_i(\bar{b},\bar{d}) \wedge \kappa(\bar{d})$. That is, \[|\varphi(M_\lambda^n,\bar{b})| - (a_{i,N_i} / a_N^{1/N}) |M_\lambda|^{N_i/N} = o(|M_\lambda|^{N_i/N}),\] whence the definition of $N$-dimensional asymptotic classes almost applies to the formula $\varphi(\bar{x},\bar{y})$, with $(\mu_i, d_i)$ being $((a_{i,N_i} / a_N^{1/N}), N_i)$ and defining formulas being ``$\exists \bar{w}[ \kappa(\bar{w}) \wedge \pi_i(\bar{y},\bar{w})]$.'' The only difference from the definition is that the defining formulas do not necessarily partition every $M_\lambda$. However they do partition all $M_\lambda$ with $|M_\lambda|$ sufficiently large. This is sufficient for the family to be an $N$-dimensional asymptotic class -- in general, an $N$-dimensional asymptotic class remains so after adding arbitrarily many structures of size at most $K$, for any fixed $K$.
\end{proof}

\end{document}